\documentclass[preprint,11pt]{elsarticle}

\usepackage{amssymb}
\usepackage{amsmath}
\usepackage{amsfonts}
\usepackage{amssymb}
\usepackage{indentfirst,latexsym,bm}
\usepackage{amsthm}
\usepackage[all]{xy}
\newtheorem{theorem}{Theorem}[section]
\newtheorem{lem}[theorem]{Lemma}

\newtheorem{prop}[theorem]{Proposition}

\theoremstyle{definition}
\newtheorem{defn}{Definition}[section]

\theoremstyle{definition}

\theoremstyle{remark}
\newtheorem{rem}{Remark}[section]
\theoremstyle{question}

\numberwithin{equation}{section}

\journal{XXX}

\begin{document}

\begin{frontmatter}



\title{The parallel sum  for adjointable operators on Hilbert $C^*$-modules }

\author[rvt]{Wei Luo}
\ead{luoweipig1@163.com}
\author[rvt]{Chuanning Song\fnref{fn1}}
%\corref{cor1}
\ead{songning@shnu.edu.cn}
\author[rvt]{Qingxiang Xu\corref{cor1}\fnref{fn1}}
\ead{qxxu@shnu.edu.cn,qingxiang\_xu@126.com}
\cortext[cor1]{Corresponding author}
\address[rvt]{Department of Mathematics, Shanghai Normal University, Shanghai 200234, PR China}
\fntext[fn1]{Supported by the
National Natural Science Foundation of China (11671261).}

\begin{abstract} The parallel sum for adjoinable operators on Hilbert $C^*$-modules is introduced and studied. Some results known for matrices and bounded linear operators on Hilbert spaces are generalized to the case of adjointable operators on Hilbert $C^*$-modules. It is shown that there exist a Hilbert $C^*$-module $H$ and two positive operators $A, B\in\mathcal{L}(H)$ such that the operator equation
$A^{1/2}=(A+B)^{1/2}X,  X\in \cal{L}(H)$ has no solution, where $\mathcal{L}(H)$ denotes the set of all adjointable operators on $H$.
\end{abstract}


\end{frontmatter}

% \linenumbers


\section{Introduction}\label{sec:introduction}
Throughout this paper, $\mathbb{C}^{m\times n}$ is the set of all $m\times n$ complex matrices. For any $A\in\mathbb{C}^{m\times n}$, let $\mathcal{R}(A)$ and $A^\dag$  denote the range and the Moore-Penrose inverse  of $A$, respectively.
The parallel sum of two Hermitian positive semi-definite matrices was first introduced by Anderson and Duffin in \cite{Anderson-Duffin}. Let $A,B\in\mathbb{C}^{n\times n}$ be Hermitian positive semi-definite matrices, the parallel sum of $A$ and $B$ is defined by \begin{equation}\label{equ:notation of parallel sum}A:B=A(A+B)^\dag B,\end{equation} which can be proved to be equal to
$A(A+B)^{-}B$ for any $\{1\}$-inverse  $(A+B)^{-}$ of $A+B$. In view of such an observation, the parallel sum was generalized by Mitra and Odell \cite{Mitra-Odell} to non-square matrices $A$ and $B$ of the same size such that
\begin{equation}\label{equ:two conditions for p.s. for matrices}\mathcal{R}(A)\subseteq \mathcal{R}(A+B)\ \mbox{and}\ \mathcal{R}(A^*)\subseteq \mathcal{R}(A^*+B^*).\end{equation}
Meanwhile, Anderson \cite{Anderson} showed that the parallel sum defined by \eqref{equ:notation of parallel sum}  can be viewed as a shorted operator for Hermitian positive semi-definite matrices $A$ and $B$. There are also many different equivalent definitions and properties of the parallel sum for matrices, the reader is referred to a recent review paper \cite{Berkics} and the references therein.

Another direction of the generalization of the parallel sum is from the finite-dimensional space to the Hilbert space. The parallel sum of positive operators $A$ and $B$ on a Hilbert space was studied by Anderson-Schreiber \cite{Anderson-Schreiber} when the range of $A+B$ is closed, and furthermore studied
in \cite{Fillmore-Williams,Morley} without any restrictions on the range of $A+B$.
As the generalizations of the parallel sum,  shorted operators and the weakly parallel sum for bounded linear operators on Hilbert spaces are also studied in \cite{Anderson-Trapp,Butler-Morley} and \cite{Antezana-Corach-Stojanoff,Djikic}, respectively.

Hilbert $C^*$-modules are generalizations of both $C^*$-algebras and Hilbert spaces, which possess some new phenomena compared with that of Hilbert spaces. For instance, a closed submodule of a Hilbert $C^*$-module may fail to be orthogonally complemented \cite[P.\,7]{Lance}, a bounded linear operator from one Hilbert $C^*$-module to another may not adjointable \cite[P.\,8]{Lance}, an adjointable operator from one Hilbert $C^*$-module to another may have no polar decomposition  \cite[Theorem~15.3.7]{Wegge-Olsen} and the famous Douglas theorem \cite[Theorem~1]{Douglas} is recently proved to be not true in the Hilbert $C^*$-module case \cite[Theorem~3.2]{Fang-Moslehian-Xu}. Another big difference will be shown by Proposition~\ref{prop:conuterexample} of this paper, which indicates that the alternative interpretation of the parallel sum initiated in \cite{Fillmore-Williams} is also no longer true in the Hilbert $C^*$-module case.

The purpose of this paper is to set up the general theory of the parallel sum for adjointable operators on Hilbert $C^*$-modules. The paper is organized as follows. Some basic knowledge about Hilbert $C^*$-modules are given in Section~\ref{sec:pre}. The term of the parallel sum is generalized to the Hilbert $C^*$-module case in Section~\ref{sec:parallel summable conditions},  where parallel summable conditions similar to \eqref{equ:two conditions for p.s. for matrices} are established.
Some properties of the parallel sum are provided in Section~\ref{sec:properties of the parallel sum}. Section~\ref{sec:norm upper bound of the parallel sum} focusses on the derivation of a formula for the norm upper bound of the parallel sum. Based on the polar decomposition and the famous Douglas theorem, an alternative interpretation of the parallel sum was introduced in \cite{Fillmore-Williams} for positive operators on a Hilbert space. It is shown in Section~\ref{sec:counterexample} that the same is not true for adjointable positive operators on Hilbert $C^*$-modules.

\section{Some basic knowledge about Hilbert $C^*$-modules}\label{sec:pre}
In this section, we recall some basic knowledge about $C^*$-algebra \cite{Pedersen}, Hilbert $C^*$-module \cite{Lance} and the Moore-Penrose inverse \cite{Xu-Chen-Song,Xu-Sheng,Xu-Zhang}. An element $a$ of a $C^*$-algebra $\mathfrak{A}$ is
said to be self-adjoint if $a=a^*$; and positive,
written $a\ge 0$, if it is self-adjoint and its spectrum $Sp(a)$
lies in $[0, +\infty)$. It is known that $a$ is positive if and only
if $a=b^*b$ for some $b\in \mathfrak{A}$ \cite[Theorem~1.3.3]{Pedersen}.

Let $a$ and $b$ be two self-adjoint elements of a $C^*$-algebra
$\mathfrak{A}$. By $b\ge a$ we mean that $b-a\ge 0$. If $a\ge 0$ and
$b\ge 0$, then $a+b\ge 0$ \cite[Theorem 1.3.3]{Pedersen}. If
furthermore $0\le a\le b$, then $\Vert a\Vert\le \Vert b\Vert$
\cite[Proposition 1.3.5]{Pedersen} and $a^t\le b^t$ for any $t\in (0,1]$ \cite[Proposition 1.3.8]{Pedersen}.

Throughout the rest of this paper, $\mathfrak{A}$ is a $C^*$-algebra. An
inner-product $\mathfrak{A}$-module is a linear space $E$ which is a right
$\mathfrak{A}$-module, together with a map $(x,y)\to \big<x,y\big>:E\times E\to
\mathfrak{A}$ such that for any $x,y,z\in E, \alpha, \beta\in \mathbb{C}$ and
$a\in \mathfrak{A}$, the following conditions hold:
\begin{enumerate}
\item[(i)] $\big<x,\alpha y+\beta
z\big>=\alpha\big<x,y\big>+\beta\big<x,z\big>;$

\item[(ii)] $\big<x, ya\big>=\big<x,y\big>a;$

\item[(iii)] $\big<y,x\big>=\big<x,y\big>^*;$

\item[(iv)] $\big<x,x\big>\ge 0, \ \mbox{and}\ \big<x,x\big>=0\Longleftrightarrow x=0.$
\end{enumerate}

Let $E$ be an inner-product $\mathfrak{A}$-module. By \cite[Proposition~1.1]{Lance}, it has
$$\big\Vert \big<x,y\big>\big\Vert\le \Vert x\Vert\,\Vert y\Vert\ \mbox{for any $x,y\in E$},$$ which induces a norm on $E$ defined by
\begin{equation}\label{equ:induced norm of Hilbert module}\Vert x\Vert=\sqrt{\Vert \big<x,x\big>\Vert}\ \mbox{for any $x\in E$}.\end{equation}
If $E$ is complete with respect to this induced norm, then $E$ is called a (right) Hilbert $\mathfrak{A}$-module. For any $F\subseteq E$, let
$\overline{F}$ denote the norm closure of $F$.

\begin{defn} A closed
submodule $M$ of a Hilbert $\mathfrak{A}$-module $H$ is said to be
orthogonally complemented  if $H=M\dotplus M^\bot$, where
$$M^\bot=\big\{x\in H\big|\langle x,y\rangle=0\ \mbox{for any }\ y\in
M\big\}.$$
In this case, the projection from $H$ onto $M$ is denoted by $P_M$.
\end{defn}

Given two Hilbert $\mathfrak{A}$-modules $H$ and $K$, let ${\cal L}(H,K)$
be the set of operators $T:H\to K$ for which there is an operator $T^*:K\to
H$ such that $$\big<Tx,y\big>=\big<x,T^*y\big> \ \mbox{for any
$x\in H$ and $y\in K$}.$$ It is known that any element $T$ of ${\cal
L}(H,K)$ must be a bounded linear operator, which is also $\mathfrak{A}$-linear
in the sense that
\begin{equation}\label{equ:module linear}T(xa)=(Tx)a \ \mbox{for any $x\in H$ and $a\in \mathfrak{A}$}.\end{equation} We call ${\cal
L}(H,K)$ the set of adjointable operators from $H$ to $K$.

For any
$T\in {\cal L}(H,K)$, the range and the null space of $T$ are denoted by
${\cal R}(T)$ and ${\cal N}(T)$, respectively. In case
$H=K$, ${\cal L}(H,H)$ which we abbreviate to ${\cal L}(H)$, is a
$C^*$-algebra whose unit, written $I_H$, is the identity operator on $H$. Let $\mathcal{L}(H)_+$ be the set of all positive elements of ${\cal L}(H)$.
For any $T\in\mathcal{L}(H)$, the notation $T\ge 0$ is used to indicate that $T$ is a positive element. In the special case that $H$ is a Hilbert space, ${\cal L}(H)$ consists of all bounded linear operators on $H$, and in this case we use the notation $\mathbb{B}(H)$ instead of ${\cal L}(H)$.

Throughout the rest of this section, $H$ and $K$ are two Hilbert
$\mathfrak{A}$-modules.

\begin{lem}\label{lem:range of positive operator of power alpha}{\rm \cite[Lemma~2.3]{Xu-Fang}} Let $T\in \mathcal{L}(H)_+$. Then $\overline{{\cal R}(T^\alpha)}=\overline{{\cal R}(T)}$ for any $\alpha\in (0,1)$.
\end{lem}

\begin{defn}
The Moore-Penrose inverse (In brief, M-P inverse) of $T\in {\cal L}(H,K)$ is denoted by  $T^\dag$, which is the
unique element $X\in {\cal L}(K,H)$ satisfying
\begin{equation} \label{equ:m-p inverse} TXT=T, \ XTX=X, \ (TX)^*=TX \ \mbox{and}\ (XT)^*=XT.\end{equation}
If such a $T^\dag$ exists, then  $T$ is said to be M-P invertible.
\end{defn}

\begin{lem}\label{lem:Xu-Sheng condition of M-P invertible}{\rm \cite[Theorem~2.2]{Xu-Sheng}}\ For any $T\in\mathcal{L}(H,K)$, $T$ is M-P invertible if and only if ${\cal R}(T)$ is closed.
\end{lem}

\begin{lem}\label{lem:orthogonal} {\rm (cf.\,\cite[Theorem 3.2]{Lance} and \cite[Remark 1.1]{Xu-Sheng})}\ Let $T\in {\cal
L}(H,K)$. Then the closeness of any one of the following sets
implies the closeness of the remaining three sets:
$${\cal R}(T), \ {\cal R}(T^*), \ {\cal R}(TT^*)\ \mbox{and}\  {\cal R}(T^*T).$$
If ${\cal R}(T)$ is closed, then ${\cal R}(T)={\cal R}(TT^*)$,
${\cal R}(T^*)={\cal R}(T^*T)$ and the following  orthogonal
decompositions hold:
\begin{equation*}\label{equ:orthogonal decomposition} H={\cal N}(T)\dotplus
{\cal R}(T^*)\ \mbox{and}\ \ K={\cal R}(T)\dotplus {\cal N}(T^*).\end{equation*}
\end{lem}

\begin{rem}\label{rem:selected properties of M-P inverse} Suppose that $T\in {\cal
L}(H,K)$ is M-P invertible. Then from \cite{Xu-Zhang} we know that $${\cal R}(T^\dag)={\cal R}(T^*), {\cal N}(T^\dag)={\cal N}(T^*), {\cal R}(T)={\cal N}(T^*)^\bot, {\cal R}(T^*)={\cal N}(T)^\bot.$$
Furthermore, it easily follows from \eqref{equ:m-p inverse} that $(T^*)^\dag=(T^\dag)^*$ and $(T^*T)^\dag=T^\dag (T^*)^\dag$.
Specifically, if $H=K$ and $T=T^*$, then $TT^\dag=T^\dag T$. If in addition $T\ge 0$, then
$T^\dag=T^\dag T T^\dag=(T^\frac12
T^\dag)^*(T^\frac12 T^\dag)\ge 0$, and $\mathcal{R}(T^\frac12)=\mathcal{R}(T)$ by Lemma~\ref{lem:orthogonal} such that
\begin{equation*}T^\dag=(T^\frac12\cdot T^\frac12)^\dag=(T^\frac12)^\dag\cdot (T^\frac12)^\dag,\end{equation*}
which means that $(T^\dag)^\frac12=(T^\frac12)^\dag$.
\end{rem}

\begin{rem}{\rm Let $T\in {\cal L}(H,K)$. By Lemma~\ref{lem:Xu-Sheng condition of M-P invertible}  we can conclude that the following two conditions are equivalent:
\begin{enumerate}
\item[{\rm (i)}] There exists  $X\in {\cal L}(K,H)$ such that $TXT=T$;
\item[{\rm (ii)}] ${\cal R}(T)$ is closed.
\end{enumerate}
}\end{rem}

\begin{defn} Let $T\in {\cal L}(H,K)$ be M-P invertible. Denote by
$$T\{1\}=\big\{X\in {\cal
L}(K,H)\,\big|\, TXT=T\big\}.$$
Any element $T^-$ of $T\{1\}$ is called a
$\{1\}$-inverse of $T$.
\end{defn}

\section{Parallel summable conditions}\label{sec:parallel summable conditions}
Throughout this section, $H, H_1,H_2,H_3$ and $K$ are Hilbert
$\mathfrak{A}$-modules unless otherwise specified.
 \begin{defn}\cite{Mitra-Odell} \label{defn:definition of parallel sum}{\rm  Let $A,B\in {\cal L}(H)$ be such that $A+B$ is M-P invertible. Then $A$ and $B$ are said to be parallel summable if
$A(A+B)^-B$ is invariant under any choice of $(A+B)^-$. In such case, the parallel sum of $A$ and $B$ is denoted by (\ref{equ:notation of parallel sum}).
}\end{defn}

Let $A,B\in\mathcal{L}(H)$ be such that $A+B$ is M-P invertible. As in the matrix case, in this section we will show that $A$ and $B$ are parallel summable if and only if \eqref{equ:two conditions for p.s. for matrices} is satisfied; see Theorem~\ref{thm:sufficient condition of parallel summable} and Remark~3.1.
In the special case that   $A$ and $B$ are both positive, we will prove in Theorem~\ref{thm:positive sum moore-penrose invertible implies parallel summable} that the M-P invertibility of $A+B$  will guarantee automatically the parallel summabillty of $A$ and $B$.

To achieve the main results of this section, we need two auxiliary lemmas as follows.

\begin{lem}\label{lem:ranges of T and TT-star}{\rm \cite[Proposition 3.7]{Lance}} It holds that $\overline{{\cal R}(TT^*)}=\overline{{\cal R}(T)}$ for any $T\in {\cal L}(H,K)$.
\end{lem}
\begin{lem}\label{lem:condition of AXB=C}{\rm \cite[Lemma~2.4]{Xu-Sheng-Gu}} Suppose that both $A\in {\cal L}(H_2, H_3)$ and $B\in {\cal L}(H_1,
H_2)$ are M-P invertible. Let  $A^-$ and $B^-$ be any $\{1\}$-inverses of $A$ and $B$, respectively. Then for any $C\in {\cal L}(H_1, H_3)$, the
equation \begin{equation} \label{equ:AXB=C} AXB=C, \  X\in {\cal
L}(H_2)\end{equation} has a solution if and only if $AA^-CB^-B=C$.
In such case, the general solution  $X$ to Eq.\,(\ref{equ:AXB=C}) is of the form
\begin{equation*}\label{eqn:a form of X with V}X=A^-CB^-+V-A^-AVBB^-,
\end{equation*}
where $V\in {\cal L}(H_2)$ is arbitrary.
\end{lem}

Now, we state the main result of this section as follows:
\begin{theorem}\label{thm:sufficient condition of parallel summable} Let $A,B\in {\cal L}(H)$ be such that $A+B$ is M-P invertible.
Then $A$ and $B$ are parallel summable if and only if
\begin{equation}\label{equ:conditions of parallel summable}A=A(A+B)^+(A+B)\ \mbox{and}\ B=(A+B)(A+B)^+B.\end{equation}
\end{theorem}
\begin{proof} Note that any $\{1\}$-inverse $(A+B)^-$ is a solution to the equation
\begin{equation}\label{equ:equation to i-inverse of A+B}(A+B)X(A+B)=A+B,\end{equation}
which guarantees the solvability of the equation above. By Lemma~\ref{lem:condition of AXB=C}, we know that the general solution $X$ to Eq.\,\eqref{equ:equation to i-inverse of A+B}
is of the form
\begin{equation*}
    X=(A+B)^\dag +Y-(A+B)^\dag (A+B)Y(A+B)(A+B)^\dag,
\end{equation*}
where $Y\in {\cal L}(H)$ is arbitrary. It follows that $A$ and $B$ are parallel summable if and only if
\begin{equation}\label{equ:demanding of Y} AYB =A(A+B)^\dag (A+B)Y(A+B)(A+B)^\dag B, \ \mbox{for any}\  Y\in {\cal L}(H).
 \end{equation}

Suppose that \eqref{equ:conditions of parallel summable} is satisfied, then the equation above is  valid obviously and thus $A$ and $B$ are parallel summable.
Conversely, suppose that $A$ and $B$ are parallel summable; or equivalently, Eq.\,(\ref{equ:demanding of Y}) is satisfied.
Then we put
\begin{equation}\label{equ:defn Y} Y=\big[I_H-(A+B)^\dag (A+B)\big]B^*.\end{equation}
Substituting such an operator $Y$ into (\ref{equ:demanding of Y}) gives
 \begin{equation}\label{equ:YB=0}A\big[I_H-(A+B)^\dag (A+B)\big]B^*B=0.
 \end{equation}
Note that \begin{equation}\label{equ:multiply A+B leads to zero}(A+B)\big[I_H-(A+B)^\dag (A+B)\big]=0,\end{equation}
so the equation above together with \eqref{equ:YB=0} yields
\begin{equation}\label{equ:middle term-B}B\big[I_H-(A+B)^\dag (A+B)\big]B^*B=0,\end{equation}
which leads clearly  to
$(YB)^*(YB)=0$, where $Y$ is defined by   \eqref{equ:defn Y}. Therefore, we have $YB=0$.
It follows that
 \begin{equation*}\big[I_H-(A+B)^\dag (A+B)\big]\xi=0\ \mbox{for any $\xi\in \overline{{\cal R}(B^*B)}$}.\end{equation*}
Since ${\cal R}(B^*)\subseteq \overline{{\cal R}(B^*)}$ and $\overline{{\cal R}(B^*)}=\overline{{\cal R}(B^*B)}$ by Lemma~\ref{lem:ranges of T and TT-star}, we have
 \begin{equation*}\big[I_H-(A+B)^\dag (A+B)\big]B^*=0.\end{equation*}
Taking $*$-operation, we get
\begin{equation}\label{equ:mulyiply B from left equals zero}B\big[I_H-(A+B)^\dag (A+B)\big]=0.\end{equation}
The first equation in (\ref{equ:conditions of parallel summable}) then follows immediately from (\ref{equ:multiply A+B leads to zero}) and (\ref{equ:mulyiply B from left equals zero}).

Similarly, the second equation in \eqref{equ:conditions of parallel summable} can be obtained  by substituting $Y=A^*\big[I_H-(A+B)(A+B)^\dag\big]$ into (\ref{equ:demanding of Y}).
\end{proof}

\begin{rem}\label{rem:one condition extended to be four conditions}{\rm Let $A,B\in {\cal L}(H)$ be such that $A+B$ is M-P invertible.
The proof of Theorem~\ref{thm:sufficient condition of parallel summable} indicates  that
\begin{eqnarray*}
                   \begin{array}{ccc}
                     A=A(A+B)^\dag (A+B) & \Longleftrightarrow  & B=B(A+B)^\dag (A+B) \\
                     \Updownarrow &  & \Updownarrow \\
                     {\cal R}(A^*)\subseteq {\cal R}(A^*+B^*) & &{\cal R}(B^*)\subseteq {\cal R}(A^*+B^*). \\
                   \end{array}
\end{eqnarray*}
Similarly, it holds that
\begin{eqnarray*}
                   \begin{array}{ccc}
                     B=(A+B)(A+B)^\dag B & \Longleftrightarrow  & A=(A+B)(A+B)^\dag A  \\
                     \Updownarrow &  & \Updownarrow \\
                     {\cal R}(B)\subseteq {\cal R}(A+B) & &{\cal R}(A)\subseteq {\cal R}(A+B). \\
                   \end{array}
\end{eqnarray*}
}\end{rem}

\begin{rem}\label{rem:two sides A or two sides B}{\rm Let $A,B\in {\cal L}(H)$ be parallel summable. Then it holds that
\begin{equation}\label{equ:two sides A or two sides B}A:B=A-A(A+B)^\dag A=B-B(A+B)^\dag B.\end{equation}
In fact, from Remark~\ref{rem:one condition extended to be four conditions} we have
\begin{equation*} A:B=(A+B)(A+B)^\dag B-B(A+B)^\dag B=B-B(A+B)^\dag B.
\end{equation*}
Similarly, \begin{equation*} A:B=A(A+B)^\dag (A+B)-A(A+B)^\dag A=A-A(A+B)^\dag A.
\end{equation*}
}\end{rem}

Next, we consider the special case that both $A$ and $B$ are positive.
\begin{lem}\label{equ:condition on positive}{\rm \cite[Lemma 4.1]{Lance}}  Let $T\in {\cal L}(H)$. Then $T\ge 0$ if and only if
 $\big<Tx, x\big>\ge 0$ for any $x\in H$.
\end{lem}

\begin{lem}\label{lem:large positive moore-penrose invertible implies large range} Let $A,B\in {\cal L}(H)$ be such that $0\le A\le B$. Then ${\cal R}(A)\subseteq {\cal R}(B)$ whenever $B$ is M-P invertible.
 \end{lem}
 \begin{proof}Since $0\le B$ and $B$ is M-P invertible, we have
 ${\cal R}(B)={\cal R}(B^*)={\cal N}(B)^\bot.$ Furthermore, from the assumption $0\le A\le B$ and Lemma~\ref{equ:condition on positive} we can get ${\cal N}(B)\subseteq {\cal N}(A)$, therefore
 \begin{eqnarray*}&&{\cal R}(A)\subseteq {\cal N}(A^*)^\bot={\cal N}(A)^\bot\subseteq {\cal N}(B)^\bot={\cal R}(B). \qedhere\end{eqnarray*}
 \end{proof}

\begin{theorem}\label{thm:positive sum moore-penrose invertible implies parallel summable} Suppose that $A,B\in \mathcal{L}(H)_+$ and $A+B$ is M-P invertible. Then $A$ and $B$ are parallel summable such that $A:B\ge 0$.
\end{theorem}
\begin{proof} Since $A,B\in \mathcal{L}(H)_+$, we have $0\le A\le A+B$. As $A+B$ is M-P invertible,
by Lemma~\ref{lem:large positive moore-penrose invertible implies large range} we have
${\cal R}(A^*)={\cal R}(A)\subseteq {\cal R}(A+B)={\cal R}(A^*+B^*)$.
It follows from Remark~\ref{rem:one condition extended to be four conditions} that (\ref{equ:conditions of parallel summable}) is satisfied, and
hence  $A$ and $B$ are parallel summable by Theorem~\ref{thm:sufficient condition of parallel summable}.

Next, we prove that $A:B\ge 0$. Indeed, since $A+B\ge 0$, we have $$(A+B)^\dag \ge 0  \ \mbox{and}\  (A+B)(A+B)^\dag=(A+B)^\dag (A+B),$$ therefore  $$(A+B)\left((A+B)^\dag\right)^\frac12=\left((A+B)^\dag\right)^\frac12 (A+B).$$
Let $T=B^{\frac12} \big((A+B)^\dag\big)^\frac12$. Then
\begin{eqnarray*}&&\Vert B^{\frac12} (A+B)^\dag B^{\frac12}\Vert=\Vert TT^*\Vert=\Vert T^*T\Vert=\left\Vert \left((A+B)^\dag\right)^\frac12  B \left((A+B)^\dag\right)^\frac12\right\Vert\\
&&\le \left\Vert \left((A+B)^\dag\right)^\frac12 (A+B) \left((A+B)^\dag\right)^\frac12\right\Vert=\Vert (A+B)(A+B)^\dag\Vert\le 1,
\end{eqnarray*}
which means that $I_H-B^{\frac12} (A+B)^\dag B^{\frac12}\ge 0$. Accordingly,
\begin{equation*} A:B=B-B(A+B)^\dag B=B^\frac12\left[I_H-B^{\frac12} (A+B)^\dag B^{\frac12}\right]B^\frac12\ge 0. \qedhere
\end{equation*}
\end{proof}

\section{Some properties of the parallel sum}\label{sec:properties of the parallel sum}
Throughout this section, $H$ is a Hilbert $\mathfrak{A}$-module. In this section, we provide some properties of the parallel sum in the general setting of adjointable operators on Hilbert $C^*$-modules.
First, we derive certain properties of $A:B$ similar to that in \cite[Section~III]{Anderson-Duffin} with  no demanding on the existence of $A^\dag$ and $B^\dag$.

\begin{prop}\label{prop:change position of A and B} {\rm (cf.\,\cite[Lemma~1]{Anderson-Duffin})} Let $A,B\in {\cal L}(H)$ be parallel summable. Then $A:B=B:A$.
\end{prop}
\begin{proof}By (\ref{equ:two sides A or two sides B}), (\ref{equ:conditions of parallel summable}) and Remakr~\ref{rem:one condition extended to be four conditions}, we have
\begin{eqnarray*}A:B&=&B-B(A+B)^\dag B=B(A+B)^\dag (A+B)-B(A+B)^\dag B\\
&=&B(A+B)^\dag A=B:A.\qedhere
\end{eqnarray*}
\end{proof}

\begin{prop}\label{prop:range of parallel sum} {\rm (cf.\,\cite[Lemma~3]{Anderson-Duffin})}  Let $A,B\in {\cal L}(H)$ be parallel summable. Then
\begin{equation}\label{equ:the range of the parallel sum}{\cal R}(A:B)={\cal R}(A)\cap {\cal R}(B).\end{equation}
\end{prop}
\begin{proof} It follows clearly from (\ref{equ:two sides A or two sides B}) that ${\cal R}(A:B)\subseteq {\cal R}(A)\cap {\cal R}(B)$. Conversely,
 given any $\xi\in {\cal R}(A)\cap {\cal R}(B)$, let $u\in A^{-1}\{\xi\}$ and $v\in B^{-1}\{\xi\}$ be chosen arbitrary. Then by (\ref{equ:two sides A or two sides B}) and (\ref{equ:conditions of parallel summable}), we have
\begin{eqnarray}&&(A:B)(u+v)=\big[A-A(A+B)^\dag A\big]u+\big[B-B(A+B)^\dag B\big]v\nonumber\\
&&=\big[I_H-A(A+B)^\dag\big]\xi+\big[I_H-B(A+B)^\dag\big]\xi\nonumber\\
\label{equ:(A:B)(invA+invB)=1}&&=\xi+\big[I_H-(A+B)(A+B)^\dag\big]\xi\nonumber\\
\label{equ:(A:B)(invA+invB)=1}&&=\xi+\big[I_H-(A+B)(A+B)^\dag\big]Bv=\xi.
\end{eqnarray}
This completes the proof of ${\cal R}(A)\cap {\cal R}(B)\subseteq {\cal R}(A:B)$. Therefore, \eqref{equ:the range of the parallel sum} is satisfied.
\end{proof}

\begin{prop}\label{prop:parallel of two projections} {\rm (cf.\,\cite[Theorem~8]{Anderson-Duffin})} Suppose that $P,Q\in\mathcal{L}(H)$ are two projections such that
$P+Q$ is M-P invertible. Then $P:Q=\frac12 P_M$, where $M=\mathcal{R}(P)\cap \mathcal{R}(Q)$.
\end{prop}
\begin{proof}From Theorem~\ref{thm:positive sum moore-penrose invertible implies parallel summable} we know that $T=2 (P:Q)\ge 0$, hence $T$ is self-adjoint. For any $x\in H$, we have $\xi=(P:Q)x\in\mathcal{R}(P)\cap \mathcal{R}(Q)$ by Proposition~\ref{prop:range of parallel sum}, hence
$\xi\in P^{-1}\{\xi\}\cap Q^{-1}\{\xi\}$ since $P$ and $Q$ are projections. It follows from \eqref{equ:(A:B)(invA+invB)=1} that
\begin{eqnarray*}T^2x=4(P:Q)\xi=2(P:Q)(\xi+\xi)=2\xi=2(P:Q)x=Tx,
\end{eqnarray*}
which means that $T$ is idempotent. Therefore, $T$ is a projection.
\end{proof}

\begin{prop}{\rm (cf.\,\cite[Lemma~6]{Anderson-Duffin})} Let  $A,B,C\in {\cal L}(H)$ be such that
both $(A:B):C$ and $A:(B:C)$ are well-defined. Then $$(A:B):C=A:(B:C).$$
\end{prop}
\begin{proof} Given any $\xi\in {\cal R}(A)\cap {\cal R}(B)\cap {\cal R}(C)$, let $u\in A^{-1}\{\xi\},v\in B^{-1}\{\xi\}$ and $w\in C^{-1}\{\xi\}$ be chosen arbitrary. Then from \eqref{equ:(A:B)(invA+invB)=1},  we have
\begin{eqnarray}&&\big[(A:B):C\big](u+v+w)=\big[(A:B):C\big](u+v)+\big[(A:B):C\big]w\nonumber\\
&&=C\big[(A:B)+C\big]^\dag (A:B)(u+v)+\Big[C-C\big[(A:B)+C\big]^\dag C\Big]w\nonumber\\
\label{equ:two sides of three operator}&&=C\big[(A:B)+C\big]^\dag\xi+\xi-C\big[(A:B)+C\big]^\dag\xi=\xi.
\end{eqnarray}

Now for any $x\in H$, let
\begin{equation}\label{equ:defn of xi}\xi=\big[A:(B:C)\big]x\in {\cal R}(A)\cap {\cal R}(B)\cap {\cal R}(C).\end{equation}
Then $\xi=\Big(A-A\big[A+(B:C)\big]^\dag A\Big)x=Au$, where
\begin{equation}\label{equ:expression of u}u=\Big(I_H-\big[A+(B:C)\big]^\dag A\Big)x\in A^{-1}\{\xi\}.
\end{equation}
Furthermore,
$$\xi=(B:C)\big[A+(B:C)\big]^\dag Ax=\big[B-B(B+C)^\dag B\big]\big[A+(B:C)\big]^\dag Ax,$$
hence
\begin{equation}\label{equ:expression of v}v=\big[I_H-(B+C)^\dag B\big]\big[A+(B:C)\big]^\dag Ax\in B^{-1}\{\xi\}.
\end{equation}
Similarly, we know that
\begin{equation}\label{equ:expression of w}w=\big[I_H-(B+C)^\dag C\big]\big[A+(B:C)\big]^\dag Ax\in C^{-1}\{\xi\}.
\end{equation}

Let $u,v$ and $w$ be given by (\ref{equ:expression of u})--(\ref{equ:expression of w}). Then by
(\ref{equ:two sides of three operator})--(\ref{equ:defn of xi}), we have
\begin{equation}\label{equ:left action equals right}\big[(A:B):C\big](u+v+w)=\big[A:(B:C)\big]x.
\end{equation}
On the other hand, by (\ref{equ:expression of u})--(\ref{equ:expression of w}) we have
\begin{equation}\label{equ:expression of u+v+w}u+v+w=x+\big[I_H-(B+C)^\dag (B+C)\big]\big[A+(B:C)\big]^\dag Ax.
\end{equation}
Note that $C\big[I_H-(B+C)^\dag (B+C)\big]=0$ and $$(A:B):C=(A:B)\Big[(A:B)+C\big]^\dag C,$$ so from (\ref{equ:expression of u+v+w})
we obtain $\big[(A:B):C\big](u+v+w)=\big[(A:B):C\big]x$. This equality together with (\ref{equ:left action equals right}) yields
$$\big[(A:B):C\big]x=\big[A:(B:C)\big]x,\ \mbox{for any $x\in H$.}$$
The asserted equality then follows from the arbitrariness of $x$ in $H$.
\end{proof}

\begin{rem}It is unknown for us  that the existence of $(A:B):C$ will imply the existence of $A:(B:C)$, and vice visa.
\end{rem}

Next, we provide two new properties. A characterization of the M-P invertibility of $A:B$ is as follows:
\begin{prop} Let $A,B\in {\cal L}(H)$ be parallel summable. Then the following two statements are equivalent:
 \begin{enumerate}
 \item[{\rm (i)}] $A$ and $B$ are both M-P invertible;
  \item[{\rm (ii)}]$A:B$ is M-P invertible.
 \end{enumerate}
\end{prop}
\begin{proof}``(i)$\Longrightarrow$(ii)": Assume that $A$ and $B$ are both M-P invertible. Then both ${\cal R}(A)$ and ${\cal R}(B)$ are closed, hence $A:B$ has a closed range by (\ref{equ:the range of the parallel sum}), therefore $A:B$ is M-P invertible by Lemma~\ref{lem:Xu-Sheng condition of M-P invertible}.

``(ii)$\Longrightarrow$(i)": Assume that $A:B$ is M-P invertible. Then for any $y\in\overline{{\cal R}(B)}$, there exists a sequence  $\{x_n\}$ in $H$ such that $Bx_n\to y$ as $n\to +\infty$. It follows that
\begin{eqnarray*}&&(A:B)x_n=Bx_n-B(A+B)^\dag Bx_n\to y-B(A+B)^\dag y\stackrel{def}{=}z.\end{eqnarray*} Since ${\cal R}(A:B)$ is closed, we have
$z\in {\cal R}(A:B)\subseteq {\cal R}(B)$ and thus $$y=z+B(A+B)^\dag y\in {\cal R}(B).$$ This completes the proof of the closeness of ${\cal R}(B)$. The proof of the closeness of ${\cal R}(A)$ is similar.
\end{proof}

As an application of the parallel, we establish a proposition as follows:
\begin{prop}Let $M$ and $N$ be two orthogonally complemented closed submodules of $H$ such that $M+N$ is closed in $H$. Then both $M+N$ and $M\cap N$ are orthogonally complemented such that $(M\cap N)^\bot=M^\bot+N^\bot$.
\end{prop}
\begin{proof} Let $P_M$ and $P_N$ be the projections from $H$ onto $M$ and $N$, respectively. Put 
\begin{align*}\label{equ:inclusion of two projectons}T=\left(
        \begin{array}{cc}
          0 & 0 \\
          P_M & P_N \\
        \end{array}\right)\in \mathcal{L}(H\oplus H).
\end{align*}
Then clearly, 
\begin{equation*}\label{equ:relationship of the ranges of two operators-1}\mathcal{R}(T)=\{0\}\oplus (M+N)\ \mbox{and}\ \mathcal{R}(TT^*)=\{0\}\oplus \mathcal{R}(P_M+P_N).
\end{equation*}
By assumption we know that $\mathcal{R}(T)$ is closed, which means by Lemma~\ref{lem:orthogonal} that $\mathcal{R}(P_M+P_N)$ is closed such that $\mathcal{R}(P_M+P_N)=M+N$.
It follows from Lemma~\ref{lem:Xu-Sheng condition of M-P invertible} that $P_M+P_N$ is M-P invertible, hence $M+N$ is orthogonally complemented such that $P_{M+N}=(P_M+P_N)(P_M+P_N)^\dag$.
Furthermore, from Proposition~\ref{prop:parallel of two projections} we know that $M\cap N$ is also orthogonally complemented such that
$P_M:P_N=\frac12 P_{M\cap N}$.

Now, we show that $(M\cap N)^\bot=M^\bot+N^\bot$. Indeed, it is obvious that $M^\bot+N^\bot\subseteq (M\cap N)^\bot$. On the other hand, for any $x\in (M\cap N)^\bot$, we have
$$P_Mx-P_M(P_M+P_N)^\dag P_Mx=(P_M:P_N)x=\frac12 P_{M\cap N}x=0,$$
hence $x_1=x-(P_M+P_N)^\dag P_Mx\in M^\bot$. Note that
\begin{equation*}P_N(x-x_1)=P_N(P_M+P_N)^\dag P_Mx=(P_M:P_N)x=0,\end{equation*}
therefore $x=x_1+(x-x_1)\in M^\bot+N^\bot$. Since $x\in (M\cap N)^\bot$ is arbitrary, we have $(M\cap N)^\bot\subseteq M^\bot+N^\bot$.
\end{proof}

\section{A formula for the norm upper bound of the parallel sum}\label{sec:norm upper bound of the parallel sum}
Let $A,B\in\mathbb{C}^{n\times n}$ be Hermitian positive semi-definite matrices. Norm upper bound (\ref{equ:sharp upper bound for parallel sum wrt positive operators}) below was
 obtained in \cite[Theorem~25]{Anderson-Duffin} with the usage of the M-P inverses $A^\dag$ and $B^\dag$. The purpose of this section is to investigate
 whether  (\ref{equ:sharp upper bound for parallel sum wrt positive operators}) is true or not for general parallel summable operators $A$ and $B$ in ${\cal L}(H)$, where $H$ is a Hilbert $\mathfrak{A}$-module. Our first result is
 Theorem~\ref{thm:sharp upper bound for parallel sum wrt positive operators}, which indicates that  (\ref{equ:sharp upper bound for parallel sum wrt positive operators}) is  true if  $A$ and $B$ are both positive. The proof of Theorem~\ref{thm:sharp upper bound for parallel sum wrt positive operators} is carried out without any use of
 $A^\dag$ and $B^\dag$, since this two M-P inverses may not exist. A counterexample is constructed in Remark~\ref{rem:countable example concerning norm upper bound A:B}, which shows that (\ref{equ:sharp upper bound for parallel sum wrt positive operators}) might be false even if $A$ and $B$ are both self-adjoint. If however, an additional condition established in Theorem~\ref{thm:generalized result of norm upper bound A:B} is satisfied, then (\ref{equ:sharp upper bound for parallel sum wrt positive operators}) is valid for such $A$ and $B$ which need not to be self-adjoint. Thus, a generalization of Theorem~\ref{thm:sharp upper bound for parallel sum wrt positive operators} is obtained.

Let $A,B\in {\cal L}(H)$ be parallel summable. Trivially, $A:B=0$ if $A=0$ or $B=0$. So throughout the rest of this section, it is assumed that both $A$ and $B$ are non-zero. Throughout the rest of this section, $H$ is a Hilbert $\mathfrak{A}$-module.

\begin{defn}{\rm For any $a,b\in (0, +\infty)$, let $a:b$ be the parallel sum of $a$ and $b$ defined by
\begin{equation}\label{equ:defn of parallel sum for numbers}a:b=\frac{ab}{a+b}.\end{equation}
}\end{defn}
\begin{theorem}\label{thm:sharp upper bound for parallel sum wrt positive operators}{\rm (cf.\,\cite[Theorem~25]{Anderson-Duffin})} Let  $A,B\in \mathcal{L}(H)_+$ be such that  $A+B$ is M-P invertible. Let $\Vert A\Vert:\Vert B\Vert$
be defined by (\ref{equ:defn of parallel sum for numbers}). Then
\begin{equation}\label{equ:sharp upper bound for parallel sum wrt positive operators}\Vert A:B\Vert\le \Vert A\Vert:\Vert B\Vert.\end{equation}
\end{theorem}
\begin{proof}By Theorem~\ref{thm:positive sum moore-penrose invertible implies parallel summable}, we know that $A$ and $B$ are parallel summable such that $A:B\ge 0$.
 Since $A\ge 0$ and $B\ge 0$, we have
\begin{equation}\label{equ:technique inequality for positive operator} A\ge \frac{A^2}{\Vert A\Vert}\ \mbox{and}\ B\ge \frac{B^2}{\Vert B\Vert}.\end{equation}
For any $x\in H$, let
$$u=\left[I_H-(A+B)^\dag A\right]x\ \mbox{and}\ v=\left[I_H-(A+B)^\dag B\right]x.$$
In view of (\ref{equ:two sides A or two sides B}) and (\ref{equ:technique inequality for positive operator}), we have
\begin{eqnarray}&&\big<u,(A:B)x\big>=\big<u,Au\big>=\big<Au,u\big>\ge \frac{1}{\Vert A\Vert}\left<A^2u,u\right>\nonumber\\
\label{eqn:bigger than inverse of norm A}&&=\frac{1}{\Vert A\Vert}\left<Au,Au\right>=\frac{1}{\Vert A\Vert}\big<(A:B)x,(A:B)x\big>.
\end{eqnarray}
Similarly, it holds that
\begin{equation}\label{eqn:bigger than inverse of norm B}\big<v,(A:B)x\big>\ge \frac{1}{\Vert B\Vert}\big<(A:B)x,(A:B)x\big>.
\end{equation}
Note that $(A:B)\big[I_H-(A+B)^\dag (A+B)\big]=0$, so by (\ref{eqn:bigger than inverse of norm A}) and (\ref{eqn:bigger than inverse of norm B}) we have
\begin{eqnarray*}&&\big<(A:B)x,x\big>=\left<(A:B)\left[I_H+I_H-(A+B)^\dag (A+B)\right]x,x\right>\\
&&=\left<\left[I_H+I_H-(A+B)^\dag (A+B)\right]x,(A:B)x\right>=\big<u+v, (A:B)x\big>\\
&&=\big<u,(A:B)x\big>+\big<v,(A:B)x\big>\ge \left(\frac{1}{\Vert A\Vert}+\frac{1}{\Vert B\Vert}\right)\big<(A:B)x,(A:B)x\big>.
\end{eqnarray*}
Thus
\begin{equation*}\big<(A:B)x,(A:B)x\big>\le (\Vert A\Vert:\Vert B\Vert)\big<(A:B)x,x\big>,
\end{equation*}
which means that
\begin{eqnarray*}&&\Vert (A:B)x\Vert^2\le (\Vert A\Vert:\Vert B\Vert) \big\Vert \big<(A:B)x,x\big>\big\Vert\le
(\Vert A\Vert:\Vert B\Vert)\Vert (A:B)x\Vert\, \Vert x\Vert,
 \end{eqnarray*}
hence
\begin{equation*}\Vert (A:B)x\Vert\le (\Vert A\Vert:\Vert B\Vert)\, \Vert x\Vert, \ \mbox{for any $x\in H$}.
\end{equation*}
This completes the proof of (\ref{equ:sharp upper bound for parallel sum wrt positive operators}).
\end{proof}

\begin{defn}{\rm For any $T\in {\cal L}(H)$, let  $|T|\in \mathcal{L}(H)_+$ and $\rho(T)$ be self-adjoint defined by
\begin{equation}\label{equ:defn of rho(T)}|T|=(T^*T)^\frac12\ \mbox{and}\ \rho(T)=\left(
            \begin{array}{cc}
              0 & T \\
              T^* & 0 \\
            \end{array}
          \right)\in {\cal L}(H\oplus H).\end{equation}
It is proved in \cite[Lemma~4.1]{Xu-Wei-Gu} that $\Vert \rho(T)\Vert=\Vert T\Vert$.
}\end{defn}

\begin{lem}\label{lem:ando}{\rm \cite[Theorem~2.3]{Ando-Hayashi}} Let $K$ be a Hilbert space and $X,Y\in \mathbb{B}(K)$. Then the triangle equality
$|X + Y| = |X| + |Y|$ holds if and only if there exists a partial isometry $U\in\mathbb{B}(K)$ such that
$X = U|X|$ and $Y = U|Y|$.
\end{lem}

\begin{lem}\label{lem:norm middle-result for Hilbert space}Let $K$ be a Hilbert space and $A,B\in \mathbb{B}(K)$ be such that $|A+B|=|A|+|B|$. Then $A+B$ is M-P invertible if and only if
$|A|+|B|$ is M-P invertible. In such case, $A$ and $B$ are parallel summable, and norm upper bound (\ref{equ:sharp upper bound for parallel sum wrt positive operators}) is valid.
\end{lem}
\begin{proof}(1) By Lemma~\ref{lem:ando},
there exists a partial isometry $U\in\mathbb{B}(K)$ such that $A=U|A|$ and $B=U|B|$. For any $x\in K$, we have
$$\big<|A|x,|A|x\big>=\big<Ax,Ax\big>=\big<U|A|x,U|A|x\big>=\big<U^*U|A|x,|A|x\big>,$$ and thus
$\big<(I_K-U^*U)|A|x,|A|x\big>=0$, which implies that $U^*U|A|=|A|$ since $I_H-U^*U$ is a projection. It follows that $U^*A=U^*U|A|=|A|$. Similarly, we have
$U^*B=|B|$. As a result, $UU^*A=U|A|=A$ and $UU^*B=U|B|=B$.

Based on the  observation above, it can be verify directly by (\ref{equ:m-p inverse}) that
$A+B$ is M-P invertible if and only if
$|A|+|B|$ is M-P invertible. In such case, it holds that
\begin{equation}\label{equ:relationship between the two M-P inverses}(|A|+|B|)^\dag=(A+B)^\dag U\ \mbox{and}\ (A+B)^\dag=(|A|+|B|)^\dag U^*.\end{equation}

(2) Suppose that  $|A|+|B|$ is M-P invertible. Then by Theorems~\ref{thm:sufficient condition of parallel summable}, \ref{thm:positive sum moore-penrose invertible implies parallel summable}
and \ref{thm:sharp upper bound for parallel sum wrt positive operators}, we have
\begin{eqnarray}&&\Sigma_1=|A|\cdot \big[I_K-(|A|+|B|)^\dag (|A|+|B|)\big]=0,\nonumber\\
&&\Sigma_2=\big[I_K-(|A|+|B|)(|A|+|B|)^\dag\big]\cdot |B|=0,\nonumber\\
\label{eqn:|A+B|=|A|+|B|-Hilbert space}&&\big\Vert |A|:|B|\big\Vert\le\big\Vert |A|\big\Vert:\big\Vert |B|\big\Vert=\Vert A\Vert:\Vert B\Vert.
\end{eqnarray}
Eq.\,(\ref{equ:relationship between the two M-P inverses}) together with the relationship between $A$ and $|A|$, $B$ and $|B|$,  gives
\begin{eqnarray*}&&A\big[I_K-(A+B)^\dag (A+B)\big]=U\Sigma_1=0,\\
&&\big[I_K-(A+B)(A+B)^\dag\big]B=U\Sigma_2=0.
\end{eqnarray*}
In other words, Eq.\,(\ref{equ:conditions of parallel summable}) is satisfied and thus $A$ and $B$ are parallel summable.
Furthermore, we have
\begin{eqnarray*}\Vert A:B\Vert=\Vert A(A+B)^\dag B\Vert=\Vert U(|A|:|B|)\Vert=\big\Vert |A|:|B|\big\Vert.
\end{eqnarray*}
The equation above together with (\ref{eqn:|A+B|=|A|+|B|-Hilbert space}) yields (\ref{equ:sharp upper bound for parallel sum wrt positive operators}).
\end{proof}

A generalization of Theorem~\ref{thm:sharp upper bound for parallel sum wrt positive operators} is as follows:
\begin{theorem}\label{thm:generalized result of norm upper bound A:B}Let  $A,B\in {\cal L}(H)$ be such that $|A+B|=|A|+|B|$ and  $A+B$ is M-P invertible. Then
$A$ and $B$ are parallel summable, and norm upper bound (\ref{equ:sharp upper bound for parallel sum wrt positive operators}) is valid.
\end{theorem}
\begin{proof}Note that ${\cal L}(H)$ is a $C^*$-algebra, so there exist  a Hilbert space $K$  and a
$C^*$-morphism $\pi: {\cal L}(H)\to \mathbb{B}(K)$ such that
$\pi$ is faithful \cite[Corollary~3.7.5]{Pedersen}. Replacing $K$ with $\pi\big(I_H)K$ if
necessary, we may assume that $\pi$ is unital, that is, $\pi(I_H)=I_K$. It is obvious that $|\pi(T)|=\pi(|T|)$ for any $T\in {\cal L}(H)$, and if
$T$ is M-P invertible in ${\cal L}(H)$, then  $\pi(T)$ is M-P invertible in $\mathbb{B}(K)$
such that $\pi(T)^\dag=\pi(T^\dag)$.
Thus, \begin{eqnarray*}|\pi(A)+\pi(B)|&=&|\pi(A+B)|=\pi(|A+B|)=\pi(|A|+|B|)\\
&=&\pi(|A|)+\pi(|B|)=|\pi(A)|+|\pi(B)|.
\end{eqnarray*}
Furthermore, $\pi(A)+\pi(B)=\pi(A+B)\in \mathbb{B}(K)$ is M-P invertible since $(A+B)\in {\cal L}(H)$ is M-P invertible.
It follows from Lemma~\ref{lem:norm middle-result for Hilbert space} that $\pi(A)$ and $\pi(B)$ are parallel summable, which means that
$\pi(\Lambda_1)=0$ and $\pi(\Lambda_2)=0$, where
\begin{eqnarray*}&&\Lambda_1=A\big[I_H-(A+B)^\dag (A+B)\big], \Lambda_2=\big[I_H-(A+B)(A+B)^\dag\big]B.
\end{eqnarray*}
As $\pi$ is faithful, we have $\Lambda_1=0$ and $\Lambda_2=0$, hence Eq.\,(\ref{equ:conditions of parallel summable}) is satisfied and thus $A$ and $B$ are parallel summable.

Finally, by Lemma~\ref{lem:norm middle-result for Hilbert space} we have
\begin{eqnarray*}\Vert A:B\Vert&=&\Vert A(A+B)^\dag B\Vert=\big\Vert \pi\big(A(A+B)^\dag B\big)\big\Vert\\
&=&\big\Vert \pi(A)\big(\pi(A+B)\big)^\dag \pi(B)\big\Vert=\big\Vert \pi(A)\big(\pi(A)+\pi(B)\big)^\dag \pi(B)\big\Vert\\
&=&\Vert \pi(A):\pi(B)\Vert\le \Vert \pi(A)\Vert:\Vert \pi(B)\Vert=\Vert A\Vert:\Vert B\Vert.\qedhere
\end{eqnarray*}
\end{proof}

\begin{rem}\label{rem:countable example concerning norm upper bound A:B}{\rm There exist Hilbert
space  $H$ and $A,B\in \mathbb{B}(H)$ such that $A=A^*, B=B^*$ and $A,B,A+B$ are all M-P invertible, whereas (\ref{equ:sharp upper bound for parallel sum wrt positive operators}) is not valid. Such an example is as follows:

Let $K$ be any Hilbert space with dim$(K)\ge 1$, and $e$ be any non-zero element of $K$. Let $P$ be the projection from $K$ onto the (closed) linear subspace spanned by $e$. Put $T=iP$ and $S=I_K-T$, where $i^2=-1$. Clearly, $T^\dag=-T$, $S^{-1}=I_K+\frac{i-1}{2}P$, $\Vert T\Vert=1$ and
\begin{eqnarray*}&&\Vert S\Vert^2=\Vert S^*S\Vert=\Vert I_K+P\Vert=2,
\end{eqnarray*}
hence $\Vert S\Vert=\sqrt{2}$. Furthermore,  $\Vert TS\Vert=\Vert (1+i)P\vert=\sqrt{2}$.

Let $H=K\oplus K$, and $A=\rho(T), B=\rho(S)\in \mathbb{B}(H)$ be defined by (\ref{equ:defn of rho(T)}). Then $$A^\dag=\left(
                                                 \begin{array}{cc}
                                                   0 & T \\
                                                   T^* & 0 \\
                                                 \end{array}
                                               \right),B^\dag=\left(
                                                 \begin{array}{cc}
                                                   0 & \big(S^{-1}\big)^*\\
                                                   S^{-1} & 0 \\
                                                 \end{array}
                                               \right), (A+B)^{-1}=\left(
                                                                     \begin{array}{cc}
                                                                       0 & I_K \\
                                                                       I_K & 0 \\
                                                                     \end{array}
                                                                   \right).
                                               $$
So
$$A:B=A(A+B)^\dag B=\left(
        \begin{array}{cc}
          0 & TS \\
          T^*S^* & 0 \\
        \end{array}
      \right)=\left(
        \begin{array}{cc}
          0 & TS \\
          S^*T^* & 0 \\
        \end{array}
      \right)=\rho(TS),
$$
hence $\Vert A:B\Vert=\Vert TS\Vert=\sqrt{2}$. As $\Vert A\Vert=\Vert T\Vert=1$ and $\Vert B\Vert=\Vert S\Vert=\sqrt{2}$, norm upper bound
(\ref{equ:sharp upper bound for parallel sum wrt positive operators}) is not valid for such $A$ and $B$.
}\end{rem}

\section{A  further remark
on the parallel sum for positive operators}\label{sec:counterexample}
In our previous sections, the parallel sum for adjointable positive operators $A$ and $B$ are defined under the  precondition that $A+B$ is M-P invertible. As shown in \cite{Fillmore-Williams}, such an additional assumption is actually redundant in the Hilbert space case. Here is the details:

Let $H$ be a Hilbert space and $T\in \mathbb{B}(H)$. By Douglas Theorem \cite[Theorem~1]{Douglas}, it holds that
\begin{equation}\label{R(T)=R((TT*)^{1/2})}{\cal R}(T)={\cal R}((TT^*)^{1/2}).
\end{equation}
Now, given any $A, B\in \mathbb{B}(H)_+$, let
\begin{equation*}T=\left(
        \begin{array}{cc}
          0 & 0 \\
          A^{1/2} & B^{1/2} \\
        \end{array}\right)\in \mathbb{B}(H\oplus H).
\end{equation*}
Direct application of (\ref{R(T)=R((TT*)^{1/2})}) yields the following proposition:
\begin{prop}Let $H$ be a Hilbert space and $A, B\in \mathbb{B}(H)_+$. Then
\begin{equation}\label{equ:rang of A squared is contained in that of squared A plus B}{\cal R}(A^{1/2})+{\cal R}(B^{1/2})={\cal R}((A+B)^{1/2}).
\end{equation}
\end{prop}

It follows from \eqref{equ:rang of A squared is contained in that of squared A plus B} and Douglas Theorem that
there are uniquely determined operators $C, D\in \mathbb{B}(H)$ such that
\begin{eqnarray*}A^{1/2}=(A+B)^{1/2}C \ \mbox{and}\ {\cal R}(C)\subseteq {\cal N}((A+B)^{1/2})^{\perp},\\
B^{1/2}=(A+B)^{1/2}D \ \mbox{and}\ {\cal R}(D)\subseteq {\cal N}((A+B)^{1/2})^{\perp}.
\end{eqnarray*}

\begin{defn}{\rm \cite[P.277]{Fillmore-Williams}} Let $H$ be a Hilbert space and $A, B\in \mathbb{B}(H)_+.$ The parallel sum of $A$ and $B$ is defined by
\begin{equation*}A:B=A^{1/2}C^*DB^{1/2}.
\end{equation*}
\end{defn}
Note that if $A+B$ is M-P invertible, then it is easy to verify that $$C=\left((A+B)^\dag\right)^\frac12\cdot A^\frac12\ \mbox{and}\ D=\left((A+B)^\dag\right)^\frac12\cdot B^\frac12,$$
so in this case we have $A^{1/2}C^*DB^{1/2}=A(A+B)^\dag B$. Thus, the term of the parallel sum $A:B$ is generalized for positive operators $A$ and $B$  without any restrictions on $\mathcal{R}(A+B)$. The proposition below indicates that the same is not true in the Hilbert $C^*$-module case.

\begin{prop}\label{prop:conuterexample}
There exist a Hilbert $C^*$-module $H$ and $A, B\in \cal{L}(H)_+$ such that the operator equation
\begin{equation}\label{equ:operator equation-no solution}A^{1/2}=(A+B)^{1/2}X, \ X\in \cal{L}(H)\end{equation}
has no solution.
\end{prop}
\begin{proof} Let $E$ be any countably infinite-dimensional Hilbert space and let $\mathbb{B}(E)$(resp.\,$\mathbb{C}(E)$) be the set of all bounded (resp.\,compact) linear operators on $E$. Given any orthogonal normalized basis $\{e_n|n\in \mathbb{N}\}$ for $E$, let $A_1, B_1\in \mathbb{C}(E)$ be defined by
\begin{eqnarray}\label{defn:counterexample A1}&&A_1(e_{2n-1})=\frac{e_{2n-1}}{2n-1} \ \mbox{and}\ A_1(e_{2n})=0,\\
\label{defn:counterexample B1}&&B_1(e_{2n-1})=0 \ \mbox{and}\ B_1(e_{2n})=\frac{e_{2n}}{2n}
\end{eqnarray}
for any $n\in \mathbb{N}$.
Let $\mathfrak{A}=\mathbb{C}(E)\dotplus \mathbb{C}I$, where $I$ is the identity operator on $E$. Then $\mathfrak{A}$ is a unital $C^*$-algebra, which itself is a Hilbert $\mathfrak{A}$-module with the inner product given by
\begin{equation}\big<X,Y\big>=X^*Y\, \mbox{for any}\, X,Y\in \mathfrak{A}.
\end{equation}

Given any $a\in \mathfrak{A}$, let $L_a:\mathfrak{A}\to \mathfrak{A}$  be the left regular representation defined by
$L_a(x)=ax$ for any $x\in \mathfrak{A}$. Then clearly, $L_a\in\mathcal{L}(\mathfrak{A})$ such that $(L_a)^*=L_{a^*}$. Conversely,
given any $T\in \cal{L}(\mathfrak{A})$, if we put $a=T(I)\in \mathfrak{A}$, then by \eqref{equ:module linear} we have
$T(x)=T(I\cdot x)=T(I)x=ax=L_a(x)$ for any $x\in \mathfrak{A}$. It follows that $\cal{L}(\mathfrak{A})\cong \mathfrak{A}$ vis left regular representation.

Now let $A=L_{A_1}$ and $B=L_{B_1}$, where $A_1$ and $B_1$ are defined by
\eqref{defn:counterexample A1} and \eqref{defn:counterexample B1}, respectively. It is obvious that $A,B\in \cal{L}(\mathfrak{A})_+$ such that
\begin{equation}\label{equ:left representation fromula for A and A+B} A^\frac12=L_{A_1^\frac12}\ \mbox{and}\ (A+B)^\frac12=L_{(A_1+B_1)^\frac12}.\end{equation} In what follows, we show that Eq.\,\eqref{equ:operator equation-no solution}
is unsolvable for such $A$ and $B$.

Suppose on the contrary that there exists $C_1=D_1+\lambda I \in \mathfrak{A}$ with $D_1\in \mathbb{C}(E)$ and $\lambda \in \mathbb{C}$ such that
$A^{1/2}=(A+B)^{1/2}L_{C_1}$. In view of \eqref{equ:left representation fromula for A and A+B} we obtain
\begin{equation}\label{ex:A1^{1/2}=(A1+B1)^{1/2}(D1+lamda I)} A_1^{1/2}=(A_1+B_1)^{1/2}(D_1+\lambda I).
\end{equation}
We may combine \eqref{defn:counterexample A1} and \eqref{defn:counterexample B1} with \eqref{ex:A1^{1/2}=(A1+B1)^{1/2}(D1+lamda I)}) to get
\begin{equation}\label{ex:D1+lambda I=P}D_1+\lambda I=P,
\end{equation}
where $P$ is the projection from $E$ onto the closed linear span of $\{ e_1,e_3,... \}$.
Since $P=P^*$, by \eqref{ex:D1+lambda I=P} we have
\begin{equation*}\label{ex:D1-D1^*}D_1-D_1^*=(\overline{\lambda}-\lambda)I\in \mathbb{C}(E),
\end{equation*}
which happens only if $\lambda=\overline{\lambda}$. Similarly, the equation $P=P^2$ gives $\lambda=0$ or $\lambda=1$.
If $\lambda=0,$ then from \eqref{ex:D1+lambda I=P} we have  $P=D_1\in \mathbb{C}(E)$, which is a contradiction since dim\big($\mathcal{R}(P)\big)=+\infty$.
If on the other hand $\lambda=1$, then $I-P=-D_1\in \mathbb{C}(E)$, which is also a contradiction.
\end{proof}


\vspace{2ex}

\begin{thebibliography}{99}

\bibitem{Anderson}W. N. Anderson, Jr., Shorted operators, SIAM J. Appl. Math. 20 (1971), 520--525.
\bibitem{Anderson-Duffin} W. N. Anderson, Jr. and R. J. Duffin, Series and parallel addition of matrices, J. Math. Anal. Appl. 26 (1969), 576--594.


\bibitem{Anderson-Schreiber}W. N. Anderson, Jr. and M. Schreiber, The infimum of two projections,  Acta Sci. Math. (Szeged) 33 (1972), 165--168.
\bibitem{Anderson-Trapp}W. N. Anderson, Jr. and G. E. Trapp, Shorted operators II, SIAM J. Appl. Math. 28 (1975), 60--71.


\bibitem{Ando-Hayashi} T. Ando and T. Hayashi, A  characterization of the operator-valued  triangle equality, J. Operator Theory 58 (2007), 463--468.

\bibitem{Antezana-Corach-Stojanoff}J. Antezana, G. Corach, and D. Stojanoff, Bilateral shorted operators and parallel sums, Linear Algebra
Appl. 2006 (414), 570--588.


\bibitem{Berkics}P. Berkics, On parallel sum of matrices, Linear Multilinear
Algebra. 65 (2017), 2114--2123.

\bibitem{Butler-Morley}C. A. Butler and T. D. Morley, A note on the shorted operator, SIAM J. Matrix Anal. Appl. 9 (1988), 147--155.

\bibitem{Djikic} M. S. Djiki$\acute{c}$, Extensions of the Fill-Fishkind formula and the infimum-parallel sum relation, Linear Multilinear
Algebra. 64 (2015), 2335--2349.


\bibitem{Douglas} R. Douglas, On majorization, factorization, and range inclusion of operators on Hilbert spaces, Proc. Amer. Math. Soc. 17 (1966), 413--415.

\bibitem{Fang-Moslehian-Xu}X. Fang, M. S. Moslehian and Q. Xu, On majorization and range inclusion of operators on Hilbert $C^*$-modules, Linear Multilinear
Algebra., to appear

\bibitem{Fillmore-Williams}P. A. Fillmore and J. P. Williams, On operator ranges, Adv. Math. 7 (1971), 254¨C281.

\bibitem{Lance}E. C. Lance, Hilbert $C^*$-modules--A toolkit for operator algebraists,
Cambridge University Press, 1995.

%\bibitem{Luo-Song-Xu}W. Luo, C. Song and Q. Xu, Perturbation estimation for the parallel sum of hermitian positive semi-definite matrices, preprint.


%\bibitem{Mitra}S. K. Mitra, Shorted operators and the identification problem, IEEE Trans. Circ. Syst. 29 (1982), 581--583.

\bibitem{Mitra-Odell}S. K. Mitra and P. L. Odell, On parallel summabillty of matrices, Linear  Algebra Appl. 74 (1986), 239--255.

%\bibitem{Mitra-Puntanen}S. K. Mitra and S. Puntanen, The shorted operator statistically interpreted, Calcutta Statist. Assoc. Bull. 40 (1990/91), 97--102.

\bibitem{Morley}T. D. Morley, An alternative approach to the parallel sum, Adv. Appl. Math. 10 (1989), 358--369.

\bibitem{Pedersen} G. K. Pedersen, $C^*$-algebras and their
automorphism groups (London Math. Soc. Monographs 14), Academic
Press, 1979.

\bibitem{Wegge-Olsen}N. E. Wegge-Olsen, $K$-theory and $C^*$-algebras: A friendly approach, Oxford Univ.
Press, Oxford, England, 1993.

\bibitem{Xu-Fang}Q. Xu and X. Fang, A note on majorization and range inclusion of adjointable operators on Hilbert $C^*$-modules, Linear  Algebra Appl. 516 (2017), 118--125.

\bibitem{Xu-Chen-Song}Q. Xu, Y. Chen and C. Song, Representations for weighted Moore-Penrose inverses of partitioned adjointable operators, Linear Algebra Appl. 438 (2013), 10--30.


\bibitem{Xu-Sheng} Q. Xu and L. Sheng, Positive semi-definite matrices of
adjointable  operators on Hilbert $C^*$-modules, Linear Algebra
Appl. 428 (2008), 992--1000.

\bibitem{Xu-Sheng-Gu} Q. Xu, L. Sheng and Y. Gu, The solutions to some operator equations, Linear Algebra
Appl. 429 (2008), 1997--2024.

\bibitem{Xu-Wei-Gu}Q. Xu, Y. Wei and Y. Gu, Sharp norm-estimations for Moore-Penrose inverses of stable perturbations of
Hilbert $C^*$-module operators, SIAM J. Numer. Anal. 47 (2010), 4735--4758.

\bibitem{Xu-Zhang}Q. Xu and X. Zhang, The generalized inverses $A^{(1,2)}_{T,S}$ of the adjointable operators on the Hilbert $C^*$-modules, J. Korean Math. Soc. 47 (2010), 363--372.



\end{thebibliography}
\end{document}